\documentclass[12pt]{amsart}
\usepackage{graphicx,textcase}
\usepackage[headings]{fullpage}
\usepackage{amssymb,epic,eepic,epsfig,amsbsy,amsmath,amscd,hyperref}
\numberwithin{equation}{section}
                        \theoremstyle{plain}
\usepackage{mathrsfs}

\newcommand\no[1]{}

\newtheorem{theorem}{Theorem}[section]
\newtheorem{thm}{Theorem}
\newtheorem{lemma}[theorem]{Lemma}

\newtheorem{proposition}[theorem]{Proposition}

\theoremstyle{definition}
\newtheorem{remark}[theorem]{Remark}

\newcommand{\lcr}{\raisebox{-5pt}{\mbox{}\hspace{1pt}
                  \epsfig{file=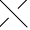}\hspace{1pt}\mbox{}}}

\def\BC{\mathbb C}

\def\BH{\mathbb H}

\def\BZ{\mathbb Z}
\def\BR{\mathbb R}
\def\BS{\mathbb S}

\def\CI{\mathcal I}

\def\fb{\mathfrak b}

\def\la{\langle}
\def\ra{\rangle}

\DeclareMathOperator{\tr}{\mathrm tr}

\def\ve{\varepsilon}
\def\be { \begin{equation} }
\def\ee { \end{equation} }

\begin{document}

\title[Volumes of two-bridge knot cone-manifolds]
{Exact integral formulas for volumes of two-bridge knot cone-manifolds}

\author[Anh T. Tran and Nisha Yadav]{Anh T. Tran and Nisha Yadav}
\address{Department of Mathematical Sciences, The University of Texas at Dallas, Richardson, TX 75080, USA}
\email{att140830@utdallas.edu}
\address{Physics Imaging Department, The University of Texas at MD Anderson Cancer Center, Houston, TX 77030, USA}
\email{fnisha@mdanderson.edu}


\begin{abstract}
We provide exact integral formulas for hyperbolic and spherical volumes of cone-manifolds whose underlying space is the $3$-sphere and whose singular set belongs to three infinite families of two-bridge knots: $C(2n,2)$ (twist knots), $C(2n,3)$, and $C(2n,-2n)$ for any non-zero integer $n$. Our formulas express volumes as integrals of explicit rational functions involving Chebyshev polynomials of the second kind, with integration limits determined by roots of algebraic equations. This extends previous work where only implicit formulas requiring numerical approximation were known. 
\end{abstract}

\maketitle

\section{Introduction}

An $n$-dimensional cone-manifold  is a simplicial complex  $M$ which can be triangulated so that the link of each simplex is piecewise-linear homeomorphic to a standard $(n-1)$-sphere and $M$ is equipped with a complete path metric such that the restriction of the metric to each simplex is isometric to a geodesic simplex of constant curvature $\kappa$. The cone-manifold is hyperbolic, Euclidean,
or spherical if  $\kappa$ is $-1$, $0$, or $+1$ respectively. 

The singular locus $\Sigma$ of a cone-manifold $M$ consists of the points in $M$ with no neighborhood isometric to a ball in a Riemannian manifold. Then $\Sigma$ is a union of totally geodesic closed simplices of dimension $n-2$. At each point of $\Sigma$ in an open $(n-2)$-simplex, there is a cone angle which is the sum of dihedral angles of $n$-simplices containing the point. In general, the cone angle may vary from point to point within a simplex. The regular set $M \setminus \Sigma$  is a dense open subset of $M$ and has a smooth Riemannian metric of constant curvature $\kappa$, but this metric is incomplete if $\Sigma \not= \emptyset$. 

In this paper, we will only consider $3$-dimensional cone-manifolds whose underlying space $M$ is the $3$-sphere $\BS^3$ and whose singular set is a knot $K$ with constant cone angle $\alpha \in [0,2\pi)$. We will denote these cone manifolds by $K(\alpha)$. 

A two-bridge knot, also known as a rational knot, is a knot that admits a projection with two maxima and two minima. In the Conway notation, a two-bridge knot corresponds to a continued fraction
\[
	[a_1, a_2, \dots, a_k] =
	a_1 + \dfrac{1}{a_2 + \dfrac{1}{\ddots + \dfrac{1}{a_k}}}
\]
and denoted by $C(a_1, a_2, \dots, a_k)$. Its diagram is shown in Figure \ref{fig:genKnotDiagram}. In the $a_i$ box, $|a_i|$ denotes the number of signed half-twists and the sign of each half-twist is the same as the sign of $a_i \in \BZ$. Here, we use the convention that  the sign of the right-handed half-twist $\lcr$ in the $a_i$ box is positive for odd $i$ and negative for even $i$. In the Schubert notation, $C(a_1, a_2, \dots, a_k)$ is the two-bridge knot $\fb(p,q)$ where $\frac{p}{q} = [a_1, a_2, \dots, a_k]$. 

\begin{figure}[h]
		\centering
		\includegraphics[scale=1]{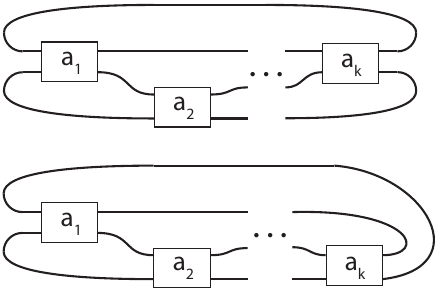}
		\caption{The rational knot $C(a_1,a_2, \ldots,a_k)$. The upper/lower one corresponds to odd/even $k$.}
		\label{fig:genKnotDiagram}
\end{figure}

\subsection*{Geometric transitions of cone-manifolds}
For any non-toric two-bridge knot $K$, Kojima \cite{Ko} and Porti \cite{Po} established that there exists a critical angle $\alpha_K \in [\frac{2\pi}{3}, \pi)$ such that the cone manifold $K(\alpha)$ undergoes geometric transitions:
\begin{itemize}
\item hyperbolic structure for $\alpha \in [0, \alpha_K)$,
\item Euclidean structure for $\alpha = \alpha_K$,
\item spherical structure for $\alpha \in (\alpha_K, 2\pi - \alpha_K)$.
\end{itemize}

\subsection*{Previous work on volume formulas}
Hilden, Lozano, and Montesinos-Amilibia \cite{HLM} introduced a method for calculating volumes of two-bridge knot cone manifolds, but without providing explicit formulas. Integral formulas for hyperbolic volumes have been obtained for specific families:
\begin{itemize}
\item $C(2n, 2)$ (twist knots) by Ham, Mednykh, and Petrov \cite{HMP},
\item $C(2n,3)$ by Ham and Lee \cite{HL},
\item $C(2n,k)$ (double twist knots $J(-2n, k)$) by Tran \cite{Tr}.
\end{itemize}
However, these formulas involve implicitly defined integrands and are primarily useful for numerical approximation. The only exact integral formulas previously known were those given by Mednykh \cite{Me} for two-bridge knots with up to seven crossings.

Note that the toric $C(2n, k)$ knots are $C(\pm 2, \mp2)$ (the trefoil knots) and $C(2n, \pm 1)$ (torus knots of type $(2,2n \pm 1)$). These knots do not admit hyperbolic structures. Their spherical volumes are computed by Mednykh \cite{Me} and Kolpakov--Mednykh \cite{KM}. 

\subsection*{Our contribution}
In this paper, we extend Mednykh's approach to obtain exact integral formulas for infinite families of two-bridge knots. We will give exact integral formulas for hyperbolic and spherical  volumes of cone-manifolds along non-toric two-bridge knots $C(2n, 2)$, $C(2n, 3)$ and $C(2n, -2n)$ where $n$ is a non-zero integer. 

To state our main result, we introduce the Chebychev polynomials of the second kind $S_k(z)$. They are recursively defined by $S_0(z)=1$, $S_1(z)=z$ and $S_{k}(z) = z S_{k-1}(z) - S_{k-2}(z)$ for all integers $k$. 

For non-toric two-bridge knots $K=C(2n,3)$, $C(2n, 2)$ or $C(2n, -2n)$ we let 
\begin{eqnarray*}
f_n(y) &=& \frac{yS_{n-1}(y) - 2S_n(y)}{(y-2)S_{n-1}(y)}, \\
g_n(y) &=& \begin{cases}
  - \displaystyle{\frac{(S_n(y) - S_{n-1}(y))^2}{(y-2)^3S^4_{n-1}(y)}}  & \quad \text{if } K=C(2n,3), \\
  - \displaystyle{\frac{S_n(y) - S_{n-1}(y)}{(y-2)^2S^3_{n-1}(y)}}  & \quad \text{if } K=C(2n,2) \text{~and~} n \not= -1, \\  
  \displaystyle{\frac{1}{(y-2)^2S^4_{n-1}(y)}} & \quad \text{if } K=C(2n, -2n) \text{~and~} |n| \ge 2.
\end{cases}
\end{eqnarray*}
The hyperbolic and spherical volumes of the cone-manifold $K(\alpha)$ are given as follows. 

\begin{thm} \label{main-h}
Let $K(\alpha)$, with $0 \le \alpha < \alpha_K$, be a hyperbolic cone-manifold. Then 
$$
\mathrm{Vol}(K(\alpha)) = i \int_{\overline{y_0}}^{y_0} \log \left( \frac{f^2_n(y)+ A^2}{(1+A^2)g_n(y)} \right) \frac{f'_n(y)}{f^2_n(y)-1} dy , $$
where $y_0$, with $\mathrm{Im} (f_n(y_0)) > 0$,  is a root of $f^2_n(y)+ A^2= (1+A^2)g_n(y)$ and $A=\cot \frac{\alpha}{2}$. The integration is taken along a simple path from $\bar{y}_0$ to $y_0$, not passing through singular points of the integrand.
\end{thm}

\begin{thm} \label{main-s}
Let $K(\alpha)$, with $\alpha_K < \alpha <  2\pi - \alpha_K$, be a spherical cone-manifold. Then, for $\alpha \in (\alpha_K, \pi]$ we have 
$$
\mathrm{Vol}(K(\alpha)) = \int_{y_+}^{y_-}  \log \left( \frac{f^2_n(y)+ A^2}{(1+A^2)g_n(y)} \right) \frac{f'_n(y)}{f^2_n(y)-1} dy,
$$
where $y_{\pm}$, with $f_n(y_{\pm}) \in \BR$,  are  roots of $f^2_n(y)+ A^2= (1+A^2)g_n(y)$ and $A=\cot \frac{\alpha}{2}$. The integration is taken along a simple path from $y_+$ to $y_-$, not passing through singular points of the integrand. 

For $\alpha \in (\pi, 2\pi - \alpha_K)$ we have 
\begin{equation} \label{>pi}
\mathrm{Vol}(K(\alpha)) = 2  \pi (\alpha - \pi) + \mathrm{Vol}(K(2\pi - \alpha)).
\end{equation} 
\end{thm}

We remark  that the equality \eqref{>pi} actually holds true for all two-bridge knots $K$. Moreover, $\mathrm{Vol}(K(\pi)) = \pi^2/p$ if $K= \fb(p,q)$ in the Schubert notation (see Proposition \ref{pi}). 

As in \cite{Me}, the proofs of Theorems \ref{main-h} and \ref{main-s} are based on  
\begin{itemize}
\item trigonometric identity between the cone angle $\alpha$ and the complex length $\gamma_\alpha$ of the singular geodesic $K$ in the cone-manifold $K(\alpha)$, and 
\item the Schl\"{a}fli formula  
$$
\kappa \, d \mathrm{Vol}(K(\alpha))  =  \frac{1}{2} l_\alpha d\alpha,
$$
where $l_\alpha = \mathrm{Re} \, \gamma_\alpha > 0$ is the real length of $K$.
\end{itemize}
This approach, based on the Schl\"afli formula, was first applied to compute volumes of hyperbolic and spherical polyhedra by Kellerhals \cite{Ke} and Vinberg \cite {Vi1, Vi2}.

The paper is organized as follows. In Section \ref{knot} we briefly review holonomy representations of hyperbolic and spherical knot cone-manifolds. In Section \ref{odd} we first study $\mathrm{SL}_2(\BC)$-representations of $C(2n, 2p+1)$, then prove trigonometric identity between the cone angle and  the complex length of the singular geodesic, and finally give a proof of Theorems \ref{main-h} and \ref{main-s}  for $C(2n,3)$. In Section \ref{even} we carry out the same things for $C(2n,2)$ and $C(2n, -2n)$. Finally, in Section \ref{ex} we present some specific examples. 

\section{Knot cone-manifolds} \label{knot}

Recall that $K(\alpha)$ denotes the $3$-dimensional cone manifolds whose underlying space $M$ is the $3$-sphere $S^3$ and whose singular set is a knot $K$ with constant cone angle $\alpha \in (0,2\pi]$. Let $G(K) := \pi_1(S^3 \setminus K)$ be the knot group, which is the fundamental group of the knot exterior. Choose the canonical
meridian-longitude pair $(\mu, \lambda)$ in $G(K)$ such
that $\mu$ is an oriented boundary of meridian disk of $K$ and $\lambda$ is null-homologous outside  $K$. 

If $K(\alpha)$ is hyperbolic, then let $\rho_\alpha: G(K) \to \mathrm{Isom}^+(\BH^3) \cong \mathrm{PSL}_2(\BC)$ be the holonomy representation. Then $\rho_\alpha$ admits two liftings to $\mathrm{SL}_2(\BC)$. 
Up to conjugation in $\mathrm{SL}_2(\BC)$, we can assume that
$$
\rho_\alpha (\mu) = \pm \left[
\begin{array}{cc}
e^{i\alpha/2} & 0 \\
0 & e^{-i\alpha/2}
\end{array}
\right],  \quad \rho_\alpha(\lambda) = \left[
\begin{array}{cc}
e^{\gamma_\alpha/2} & 0\\
0 & e^{-\gamma_\alpha/2}
\end{array}
\right]
$$
where $\gamma_\alpha = l_\alpha + i \varphi_\alpha$, $l_\alpha$ is the length of $K$, and $\varphi_\alpha \in [-2\pi, 2\pi)$ is the angle of the lifted holonomy of $K$. We call $\gamma_\alpha = l_\alpha + i \varphi_\alpha$ the complex length of the singular geodesic $K$. 

If $K(\alpha)$ is spherical, then let $\rho_\alpha: G(K) \to \mathrm{Isom}^+(\BS^3) \cong\mathrm{SO}(4)$ be the holonomy representation. Then $\rho_\alpha$ admits two liftings to $\mathrm{SU}(2) \times \mathrm{SU}(2)$. 
Up to conjugation in $\mathrm{SU}(2) \times \mathrm{SU}(2)$, we can assume that
\begin{eqnarray*}
\rho_\alpha(\mu) &=& \left( \pm \left[
\begin{array}{cc}
e^{i\alpha/2} & 0 \\
0 & e^{-i\alpha/2}
\end{array}
\right], \pm \left[
\begin{array}{cc}
e^{i\alpha/2} & 0 \\
0 & e^{-i\alpha/2}
\end{array}
\right] \right),  \\
\rho_\alpha(\lambda) &=& \left( \left[
\begin{array}{cc}
e^{i\gamma} & 0 \\
0 & e^{-i\gamma}
\end{array}
\right],  \left[
\begin{array}{cc}
e^{i\phi} & 0 \\
0 & e^{-i\phi}
\end{array}
\right]\right).
\end{eqnarray*}
In this case $l_\alpha = \gamma - \phi$ is the length of the knot $K$, and $\varphi_\alpha = \gamma + \phi \in [-2\pi, 2\pi)$ is the angle of the lifted holonomy of $K$. Note that $\gamma = \frac{1}{2}(\varphi_\alpha + l_\alpha)$ and $\phi = \frac{1}{2}(\varphi_\alpha - l_\alpha)$. 

For any non-toric  two-bridge knot $K$, by \cite{Ko, Po} there exists an angle $\alpha_K \in [\frac{2\pi}{3}, \pi)$ such that the cone manifold $K(\alpha)$ is hyperbolic if $\alpha \in [0, \alpha_K)$, Euclidean if $\alpha = \alpha_K$ and spherical if $\alpha \in (\alpha_K, 2\pi - \alpha_K)$. 

\begin{proposition} \label{pi}
Let $K=\fb(p,q)$ be a  two-bridge knot in the Schubert notation. Then
\[
  \mathrm{Vol}(K(\pi)) = \frac{\pi^2}{p}.
\]
Moreover,  for $\alpha \in (\pi, 2\pi - \alpha_K)$ we have 
$$
\mathrm{Vol}(K(\alpha)) = 2  \pi (\alpha - \pi) + \mathrm{Vol}(K(2\pi - \alpha)).
$$
\end{proposition}

\begin{proof}
The two-fold branched cover of $S^3$ branched along $K=\fb(p,q)$ with cone angle $\pi$ is
the lens space $L(p,q)$, which admits a $p$-fold cover by $S^3$. Since
$\mathrm{Vol}(S^3)=2\pi^2$, we get $\mathrm{Vol}(L(p,q))=2\pi^2/p$, and hence
$\mathrm{Vol}(K(\pi))=\frac{1}{2}\mathrm{Vol}(L(p,q))=\pi^2/p$.

For $\alpha \in (\pi, 2\pi - \alpha_K)$ we use the following arguments from \cite{MR, Po, Me}. First of all, we have $l_\alpha = 4\pi - l_{2\pi -\alpha}$. This formula was first proved by Mednykh and Rasskazov \cite{MR} for the figure eight knot, and then by Porti \cite{Po} for all two-bridge knots. Note that in \cite{MR} the notation $l_\alpha$ is used for half of the length of $K(\alpha)$; whereas in \cite{Po}  it is used for the entire length.

By the Schl\"{a}fli formula we get $\frac{d}{d \alpha} \mathrm{Vol}(K(\alpha)) = \frac{1}{2} l_\alpha$. Then 
$$
\frac{d}{d \alpha} \mathrm{Vol}(K(2\pi - \alpha)) = - \frac{1}{2} l_{2\pi -\alpha} = -2\pi + \frac{1}{2} l_\alpha.
$$
This implies that the derivative of the function $F(\alpha) := \mathrm{Vol}(K(\alpha)) - \mathrm{Vol}(K(2\pi - \alpha))$ with respect to $\alpha$ is equal to $2\pi$. Since $F(\pi) = 0$ we obtain 
$$ 
\mathrm{Vol}(K(\alpha)) - \mathrm{Vol}(K(2\pi - \alpha)) = 2\pi (\alpha - \pi)
$$
for all $\alpha$. Note that the last relation is exactly the statement of Lemma 4.2 in \cite{Po}.
\end{proof}

\section{$C(2n, 2p+1)$}
\label{odd}

\subsection{Knot group}

\begin{proposition} \label{group-odd}
We have $G(C(2n, 2p+1)) = \la a, b \mid \omega a = b \omega\ra$ where 
$$\omega = (ab)^n [(a^{-1}b^{-1})^{n}(ab)^n]^p.$$
\end{proposition}

\begin{figure}[h]
		\centering
		\includegraphics[scale=1]{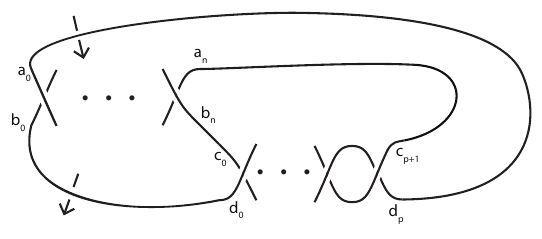}
		\caption{$C(2n,2p+1)$}
\end{figure}

\begin{proof}
Starting from the left hand section of $2n$ crossings, by induction we have 
	\begin{align*}
		a_k &= (b_0 a_0)^{-k} a_0 (b_0 a_0)^{k}, \\
		b_k &= (b_0 a_0)^{-k} b_0 (b_0 a_0)^{k}.
	\end{align*}
Similarly, in the right hand section of $l$ crossings we have
	\begin{align*}
		c_k &= (c_0^{-1}d_0)^{-k} c_0 (c_0^{-1}d_0)^{k}, \\
		d_k &= (c_0^{-1}d_0)^{-k} d_0 (c_0^{-1}d_0)^{k}.
	\end{align*}

By using the identity $a_0 = c_{p+1}$ we have
\begin{eqnarray*}
&& a_0 = (c_0^{-1}d_0)^{-p-1} c_0 (c_0^{-1}d_0)^{p+1} \\
&\Longrightarrow& a_0 = (b_n^{-1}b_0)^{-p-1} b_n (b_n^{-1}b_0)^{p+1} \\
&\Longrightarrow& a_0 = (b_n^{-1}b_0)^{-p-1} (b_0 a_0)^{-n} b_0 (b_0 a_0)^{n} (b_n^{-1}b_0)^{p+1} \\
&\Longrightarrow& (b_0 a_0)^{n} (b_n^{-1}b_0)^{p+1} a_0 = b_0 (b_0 a_0)^{n} (b_n^{-1}b_0)^{p+1}.
\end{eqnarray*}
Hence $\omega a_0 = b_0 \omega$ where $\omega = (b_0 a_0)^{n} (b_n^{-1}b_0)^{n+1}$.

Let $a=a_0$ and $b = b_0$.  Then $b_n^{-1}b_0 = (b a)^{-n} b^{-1} (b a)^{n}b = (b a)^{-n} (ab)^n$. Hence
$$
\omega = (b_0 a_0)^{n} (b_n^{-1}b_0)^{p+1}= (ba)^n [(b a)^{-n} (ab)^n]^{p+1} = (ab)^n [(b a)^{-n} (ab)^n]^p. 
$$
This completes the proof. 
\end{proof}

Note that the knot group presentation  in Proposition \ref{group-odd} is different from the one in \cite{HS, MPL, MT} (where $C(2n, 2p+1)$ is denoted by $J(-2n, 2p+1)$), but it can be applied to find exact integral formulas for volumes of cone-manifolds along $C(2n, 2p+1)$. 

\subsection{$\mathrm{SL}_2(\BC)$-representations}

Suppose $\rho\colon G(C(2n, 2p+1))\to  \mathrm{SL}_2(\BC)$ is a nonabelian representation. 
Up to conjugation, we may assume that 
\begin{equation} \label{repn}
A:=\rho(a) = \left[
\begin{array}{cc}
m & 1\\
0 & m^{-1}
\end{array}
\right] \quad \text{and} \quad B:=\rho(b) = \left[
\begin{array}{cc}
m & 0\\
y - m^2 - m^{-2} & m^{-1}
\end{array}
\right]
\end{equation}
where $(m,y) \in \BC^2$ satisfies $\rho(\omega a) = \rho(b \omega)$. Note that $y=\tr \rho(ab)$. 

We now solve the matrix equation  $\rho(\omega a) = \rho(b \omega)$. Recall that $S_k(z)$'s are the Chebychev polynomials defined by $S_0(z)=1$, $S_1(z)=z$ and $S_{k}(z) = z S_{k-1}(z) - S_{k-2}(z)$ for $k \in \BZ$. Note that $S_k(z) = (s^{k+1} - s^{-k-1})/(s - s^{-1})$ if $z = s + s^{-1}$.

The following lemmas are elementary, see e.g. \cite{MT} and references therein. 

\begin{lemma} \label{chev}
For any integer $k$ we have 
$$S^2_k(z) - z S_{k}(z) S_{k-1}(z) + S^2_{k-1}(z) =1.$$
\end{lemma}

\begin{lemma} \label{power}
Suppose $M  \in \mathrm{SL}_2(\BC)$ and $z= \tr M.$ For any integer $k$ we have 
$$
M^k = S_k(z) I - S_{k-1}(z) M^{-1}. 
$$
\end{lemma}

Let $x= \tr \rho(a)=\tr \rho(b) = m+m^{-1}$. Let  $U=\rho((a^{-1} b^{-1})^{n} (ab)^{n})$ and $u = \tr U$. 

\begin{proposition} \label{u}
	We have
	\[
		u = 2 + (y-2) (y+2-x^2) S^2_{n-1}(y).
	\]
\end{proposition}

\begin{proof}
	Recall that $A = \rho(a)$ and $B= \rho(b)$. Since $\tr A^{-1} B^{-1} = y$ and $\tr AB = y$, by Lemma \ref{power} we have
	\begin{align*}
		U &= (A^{-1} B^{-1})^{n} (AB)^{n} \\
		&= (S_n(y)I - S_{n-1}(y)BA)  (S_n(y) I - S_{n-1}(y) B^{-1}A^{-1})\\
		&= S_n^2(y) I + S_{n-1}^2(y) BA B^{-1}A^{-1} - S_n(y)S_{n-1}(y) (B^{-1}A^{-1}  + BA).
	\end{align*}
	Taking trace we obtain
	$$
	u= \tr U = 2S_n^2(y) + (\tr BA B^{-1}A^{-1}) S_{n-1}^2(y)  - 2y S_n(y)S_{n-1}(y).
	$$
	
	By Lemma \ref{chev} we have $S_n^2(y) - y S_n(y) S_{n-1}(y)  + S_{n-1}^2(y) =1$. This implies that 
	$$
	u =  2 +  (\tr BA B^{-1}A^{-1}-2) S_{n-1}^2(y).
	$$
Finally, by a direct calculation using the matrix form \eqref{repn} we have	
$$
\tr BA B^{-1}A^{-1}-2 = (y-2) (y - m^2 - m^{-2}) = (y-2) (y+2-x^2).
$$ 
Hence $u = 2 + (y-2) (y+2-x^2) S^2_{n-1}(y)$.
\end{proof}

\begin{proposition} \label{R}
We have 
$$
\rho(\omega a)  -  \rho(b \omega) =  
\left[
\begin{array}{cc}
0 & \Phi_{C(2n, 2p+1)}(x,y)\\
(x^2-2-y) \Phi_{C(2n, 2p+1)}(x,y) & 0
\end{array}
\right]
$$
where 
$$ 
\Phi_{C(2n, 2p+1)}(x,y) = (S_n(y) - S_{n-1}(y)) S_p(u)  - (S_{n-1}(y) - S_{n-2}(y)) S_{p-1}(u).
$$
\end{proposition}

\begin{proof}
Let $W= \rho(\omega)$. Then $W = \rho((ab)^n [(b a)^{-n} (ab)^n]^p)= (AB)^n U^p$. Since $\tr U = u$, by Lemma \ref{power} we have
\begin{eqnarray*}
W &=& (AB)^n (S_p(u) I - S_{p-1}(u) U^{-1}) \\
&=& S_p(u) (AB)^n - S_{p-1}(u) (BA)^n \\
&=& S_p(u) (S_n(y) I - S_{n-1}(y) B^{-1}A^{-1}) - S_{p-1}(u) (S_n(y) I - S_{n-1}(y) A^{-1}B^{-1}).
\end{eqnarray*}
Hence 
\begin{eqnarray*}
WA - BW &=& S_p(u) [S_n(y) (A-B) - S_{n-1}(y) (B^{-1}-A^{-1})] \\
&& - \, S_{p-1}(u) [S_n(y) (A-B)  - S_{n-1}(y) (A^{-1}B^{-1}A - B A^{-1}B^{-1})].
\end{eqnarray*}
By direct calculations using the matrix form \eqref{repn} we have	
\begin{eqnarray*}
A - B &=& \left[
\begin{array}{cc}
0 & 1\\
m^2 + m^{-2} - y & 0
\end{array}
\right], \\
B^{-1}-A^{-1} &=& \left[
\begin{array}{cc}
0 & 1\\
m^2 + m^{-2} - y & 0
\end{array}
\right], \\
A^{-1}B^{-1}A - B A^{-1}B^{-1} &=& \left[
\begin{array}{cc}
0 & y-1\\
(y-1)(m^2 + m^{-2} - y) & 0
\end{array}
\right]. 
\end{eqnarray*}
Hence $WA - B W = \left[
\begin{array}{cc}
0 & \Phi\\
(m^2 + m^{-2} -y) \Phi & 0
\end{array}
\right]$ where 
$$
\Phi = S_p(u) (S_n(y)  - S_{n-1}(y)) - S_{p-1} (u) (S_n(y)  - (y-1)S_{n-1}(y)).
$$
Finally, since $S_n(y)  - (y-1)S_{n-1}(y) = S_{n-1}(y) - S_{n-2}(y)$ the proposition follows.
\end{proof}

Proposition \ref{R} implies that $\rho(\omega a) = \rho(b \omega)$ if and only if $\Phi_{C(2n, 2p+1)}(x,y) =0$.

\begin{remark}
The polynomial $\Phi_{C(2n, 2p+1)}(x,y)$ is called the Riley polynomial of the two-bridge knot $C(2n, 2p+1)$, see \cite{Ri}.
\end{remark}

\subsection{Longitude and trignometric identity}

If we choose the meridian to be $\mu=b$ then the canonical longitude is $\lambda = \omega \omega^* b^{-4n}$, where $\omega^*$ is the word obtained from $\omega$ by writing the letters in $\omega$ in reversed order. Note that since $\omega a = b \omega$ and $\omega^* b = a \omega^*$ we have $\omega \omega^* b = \omega a \omega^* = b \omega \omega^*$. 

Since $\rho(\mu) = \left[
\begin{array}{cc}
m & 0\\
* & m^{-1}
\end{array}
\right]$ we have $\rho(\lambda) = \left[
\begin{array}{cc}
l & 0\\
* & l^{-1}
\end{array}
\right]$. By a similar calculation as in \cite[Equation (2.3)]{HS} we have $lm^{4n} = - W_{12} \big/ \widetilde{W_{12}}$, where $W_{12}$ is the $(1,2)$-entry of $W=\rho(\omega)$ and $\widetilde{W_{12}}$ is obtained from $W_{12}$ by replacing $m$ by $m^{-1}$. Note that $W_{12}$ is a function in $m$ and $y$. 

\begin{proposition} \label{12}
We have 
$$
W_{12} = \left(m^{-1} - m \frac{S_n(y) - S_{n-1}(y)}{S_{n-1}(y) - S_{n-2}(y)} \right) S_{p}(u)S_{n-1}(y). 
$$
\end{proposition}

\begin{proof}
From the proof of Proposition \ref{R} we have
$$
W = S_p(u) (S_n(y) I - S_{n-1}(y) B^{-1}A^{-1}) - S_{p-1}(u) (S_n(y) I - S_{n-1}(y) A^{-1}B^{-1}).
$$
Taking the $(1,2)$-entry we have 
$$
W_{12} = - S_p(u) S_{n-1}(y) (B^{-1}A^{-1})_{12} + S_{p-1}(u) S_{n-1}(y) (A^{-1}B^{-1})_{12}.
$$
Since $(B^{-1}A^{-1})_{12} = - m^{-1}$ and $(A^{-1}B^{-1})_{12} = -m$, we obtain 
$$
W_{12} = (m^{-1}  S_p(u) - m S_{p-1}(u)) S_{n-1}(y).
$$

We now simplify $W_{12}$ by using $\Phi_{C(2n, 2p+1)}(x,y) =0$. Since $(S_n(y) - S_{n-1}(y)) S_p(u)  - (S_{n-1}(y) - S_{n-2}(y)) S_{p-1}(u)=0$,  we have $S_{p-1}(u) = \frac{S_n(y) - S_{n-1}(y)}{S_{n-1}(y) - S_{n-2}(y)} S_p(u)$. Hence
$$
W_{12} = \left(m^{-1} - m \frac{S_n(y) - S_{n-1}(y)}{S_{n-1}(y) - S_{n-2}(y)} \right) S_{p}(u)S_{n-1}(y)
$$
as claimed. 
\end{proof}

\begin{proposition} \label{fn}
Let $f_n(y) = \frac{yS_{n-1}(y) - 2S_n(y)}{(y-2)S_{n-1}(y)}$. Then if $\Phi_{C(2n, 2p+1)}(x,y) =0$, where $x=m+m^{-1}$, we have 
$$
 f_n(y) = - \frac{\ell^{1/2} +\ell^{-1/2}}{\ell^{1/2}-\ell^{-1/2}} \cdot \frac{m+m^{-1}}{m-m^{-1}}
$$ 
for $\ell = lm^{4n}$.
\end{proposition}

\begin{proof} 
Since $\ell = lm^{4n}= - W_{12} \big/ \widetilde{W_{12}}$, by Proposition \ref{12} we have
$$
\ell = - \frac{m^{-1} - m r} {m - m^{-1} r}=  \frac{m^2 r-1}{m^2 - r}
$$
where $r =\frac{S_n(y) - S_{n-1}(y)}{S_{n-1}(y) - S_{n-2}(y)}$.
This implies that 
$$
\frac{\ell+1}{\ell-1} = \frac{(m^2-1)(r+1)}{(m^2+1)(r-1)}.
$$
Hence 
$$
\frac{\ell+1}{\ell-1} \cdot \frac{m^2+1}{m^2-1} = \frac{r+1}{r-1} =  \frac{S_n(y)-S_{n-2}(y)}{(y-2)S_{n-1}(y)} = - f_n(y). 
$$
\end{proof}

\underline{Hyperbolic case}: Let $K(\alpha)$ be a hyperbolic 3-dimensional cone-manifold whose singular set is $K=C(2n, 2p+1)$ with cone angle $\alpha \in [0, 2\pi)$. Up to conjugation in $\mathrm{SL}_2(\BC)$, 
$$
\rho_\alpha (\mu) = \pm \left[
\begin{array}{cc}
e^{i\alpha/2} & 0 \\
0 & e^{-i\alpha/2}
\end{array}
\right],  \quad \rho_\alpha(\lambda) = \left[
\begin{array}{cc}
e^{\gamma_\alpha/2} & 0\\
0 & e^{-\gamma_\alpha/2}
\end{array}
\right]
$$
where $\gamma_\alpha = l_\alpha + i \varphi_\alpha$ is the complex length of the singular geodesic $K$ in $K(\alpha)$, $l_\alpha > 0$ is the real length of $K$, and $\varphi_\alpha \in [-2\pi, 2\pi)$ is the angle of the lifted holonomy of $K$. 

\begin{proposition} \label{trig} In the hyperbolic case we have 
$$
 i \coth \left( \frac{\gamma_\alpha +  4n i  \alpha}{4} \right) \cot \left( \frac{\alpha}{2} \right) = f_n(y).
$$ 
In particular, we have $\mathrm{Im} (f_n(y)) > 0$.
\end{proposition}

\begin{proof} 
Since $m = e^{i \alpha/2}$ and $\ell = lm^{4n}  = e^{(\gamma_\alpha + 4ni \alpha)/2}$, we obtain
\begin{eqnarray*}
f_n(y) &=& - \frac{\ell^{1/2} +\ell^{-1/2}}{\ell^{1/2}-\ell^{-1/2}} \cdot \frac{m+m^{-1}}{m-m^{-1}} \\
&=& i \coth \left( \frac{\gamma_\alpha +  4n i  \alpha}{4} \right) \cot \left( \frac{\alpha}{2} \right).
\end{eqnarray*} 
Note that $\cot \left( \frac{\alpha}{2} \right) > 0$ and $\mathrm{Re} (\gamma_\alpha +  4n i  \alpha) = l_\alpha>0$. Hence 
$$
\mathrm{Re}(-i f_n(y)) =  \cot \left( \frac{\alpha}{2} \right) \mathrm{Re} \coth \left( \frac{\gamma_\alpha +  4n i  \alpha}{4} \right) >0.
$$
This implies that $\mathrm{Im} (f_n(y)) > 0$.
\end{proof}

\underline{Spherical case}: Let $K(\alpha)$ be a spherical 3-dimensional cone-manifold whose singular set is $K=C(2n, 2p+1)$ with cone angle $\alpha \in [0, 2\pi)$. Up to conjugation in $\mathrm{SU}(2) \times \mathrm{SU}(2)$, we can assume that
\begin{eqnarray*}
\rho_\alpha(\mu) &=& \left( \pm \left[
\begin{array}{cc}
e^{i\alpha/2} & 0 \\
0 & e^{-i\alpha/2}
\end{array}
\right], \pm \left[
\begin{array}{cc}
e^{i\alpha/2} & 0 \\
0 & e^{-i\alpha/2}
\end{array}
\right] \right),  \\
\rho_\alpha(\lambda) &=& \left( \left[
\begin{array}{cc}
e^{i\gamma} & 0 \\
0 & e^{-i\gamma}
\end{array}
\right],  \left[
\begin{array}{cc}
e^{i\phi} & 0 \\
0 & e^{-i\phi}
\end{array}
\right]\right).
\end{eqnarray*}
In this case $l_\alpha = \gamma - \phi$ is the length of the knot $K$, and $\varphi_\alpha = \gamma + \phi \in [-2\pi, 2\pi)$ is the angle of the lifted holonomy of $K$. Note that $\gamma = \frac{1}{2}(\varphi_\alpha + l_\alpha)$ and $\phi = \frac{1}{2}(\varphi_\alpha - l_\alpha)$. Hence $m = e^{i\alpha/2}$ and $\ell=e^{i(\varphi_\alpha \pm l_\alpha)/2}$. 

\begin{proposition} \label{trig-spherical} In the spherical case we have 
$$
\cot \left( \frac{\varphi_\alpha \pm l_\alpha+  4n  \alpha}{4} \right) \cot \left( \frac{\alpha}{2} \right) = f_n(y_{\pm}).
$$ 
In particular, we have $f_n(y_{\pm}) \in \BR$.
\end{proposition}

\begin{proof} 
Since $m = e^{i \alpha/2}$ and $\ell = lm^{4n}  = e^{i(\varphi_\alpha \pm l_\alpha + 4n\alpha)/2}$, we obtain
\begin{eqnarray*}
f_n(y_{\pm}) &=& - \frac{\ell^{1/2} +\ell^{-1/2}}{\ell^{1/2}-\ell^{-1/2}} \cdot \frac{m+m^{-1}}{m-m^{-1}} \\
&=& \cot \left( \frac{\varphi_\alpha \pm l_\alpha+  4n  \alpha}{4} \right) \cot \left( \frac{\alpha}{2} \right).
\end{eqnarray*} 
\end{proof}

\subsection{Proof of Theorems \ref{main-h} and \ref{main-s} for $C(2n,3)$}

Suppose $p=1$. By Propositions \ref{u} and \ref{R} we have $u = 2 + (y-2) (y+2-x^2) S^2_{n-1}(y)$ and 
\begin{eqnarray*}
\Phi_{C(2n,3)}(x,y) &=& (S_n(y) - S_{n-1}(y)) u - (S_n(y) - (y-1)S_{n-1}(y)) \\
&=&  (y-2) (y+2-x^2) S^2_{n-1}(y) (S_n(y) - S_{n-1}(y)) \\
&& + \, 2 (S_n(y) - S_{n-1}(y))  - (S_n(y) - (y-1)S_{n-1}(y)).
\end{eqnarray*} 

\begin{lemma} \label{cd}
Suppose $\Phi(x,y) =  b - a x^2$ where $a, b \in \BC(y)$. Let $A = \cot \frac{\alpha}{2}$. Then, for any $c \in  \BC(y)$, the equation $\Phi(2\cos \frac{\alpha}{2}, y)=0$ is equivalent to $c^2+ A^2= (1+A^2)d$ where $d = 1 + (c^2-1)(1 - \frac{b}{4a})$. 
\end{lemma}

\begin{proof}
Let $x = 2\cos \frac{\alpha}{2}$. We have $4-x^2 = 4 \sin^2 \frac{\alpha}{2} = \frac{4}{A^2+1}$. Hence
\begin{eqnarray*}
\Phi(x, y)=0 &\Longleftrightarrow& x^2 =\frac{b}{a} \\
&\Longleftrightarrow& \frac{4}{A^2+1} = 4 - \frac{b}{a} = \frac{4(d-1)}{c^2 - 1} \\
&\Longleftrightarrow& (A^2 + 1) (d-1) = c^2 -1 \\
&\Longleftrightarrow& c^2+ A^2= (1+A^2)d. 
\end{eqnarray*} 
This proves the lemma.
\end{proof}

We write $\Phi_{C(2n,3)}(x,y) = b - a x^2$, where 
\begin{eqnarray*}
a &=& (y-2) S^2_{n-1}(y) (S_n(y) - S_{n-1}(y)), \\
b &=& S_n(y) + (y-3) S_{n-1}(y) + (y-2) (y+2) S^2_{n-1}(y) (S_n(y) - S_{n-1}(y)).
\end{eqnarray*} 
Note that $b - 4a = S_n(y) + (y-3) S_{n-1}(y) + (y-2)^2 S^2_{n-1}(y) (S_n(y) - S_{n-1}(y))$.

Choose $c = f_n(y)$. By Lemma \ref{cd}, the equation $\Phi_{C(2n,3)}(2\cos \frac{\alpha}{2}, y)=0$ is equivalent to $c^2+ A^2= (1+A^2)d$ where $d = 1 + (c^2-1)(1 - \frac{b}{4a})$. Since 
$$
c^2 -1 = \frac{4(S_n(y)-S_{n-1}(y)) (S_n(y) - (y-1) S_{n-1}(y))}{(y-2)^2 S^2_{n-1}(y)},
$$ 
we have
\begin{eqnarray*}
(c^2-1) \left( 1 - \frac{b}{4a} \right)
&=& -  \frac{S_n(y) - (y-1) S_{n-1}(y)}{(y-2)^3 S^4_{n-1}(y)} [ S_n(y) + (y-3) S_{n-1}(y) \\
&& \qquad \qquad \qquad \qquad + (y-2)^2 S^2_{n-1}(y) (S_n(y) - S_{n-1}(y)) ] \\
&=& -  \frac{1}{(y-2)^3 S^4_{n-1}(y)}  [ (S_n(y) - S_{n-1}(y))^2 - (y-2)^2 S^2_{n-1}(y) \\
&&  + \, (y-2)^2 S^2_{n-1}(y) \left( S^2_n(y) - y S_n(y) S_{n-1}(y) + (y-1) S^2_{n-1}(y) \right) ].
\end{eqnarray*}

By Lemma \ref{chev} we have $S^2_n(y) - y S_n(y) S_{n-1}(y) +  S^2_{n-1}(y) = 1$. This implies that $S^2_n(y) - y S_n(y) S_{n-1}(y) + (y-1) S^2_{n-1}(y) = 1 +  (y-2) S^2_{n-1}(y)$. Hence 
\begin{eqnarray*}
(c^2-1) \left( 1 - \frac{b}{4a} \right)
&=& -  \frac{(S_n(y) - S_{n-1}(y))^2 + (y-2)^3 S^4_{n-1}(y)}{(y-2)^3 S^4_{n-1}(y)}
\end{eqnarray*}
and  $d = 1 + (c^2-1) \left( 1 - \frac{b}{4a} \right) = -  \frac{(S_n(y) - S_{n-1}(y))^2}{(y-2)^3 S^4_{n-1}(y)}$. 

In summary, we have proved the following. 

\begin{proposition} \label{Riley-odd}
Let 
$$
f_n(y) = \frac{yS_{n-1}(y) - 2S_n(y)}{(y-2)S_{n-1}(y)}, \qquad g_n(y) = - \frac{(S_n(y) - S_{n-1}(y))^2}{(y-2)^3S^4_{n-1}(y)}.
$$
Let $x = 2\cos \frac{\alpha}{2}$ and $A = \cot \frac{\alpha}{2}$. Then the equation $\Phi_{C(2n, 3)}(x,y) = 0$ is equivalent to $f^2_n(y)+ A^2= (1+A^2)g_n(y)$.
\end{proposition}

For a two-bridge knot $K$, there exists an angle $\alpha_K \in [\frac{2\pi}{3},\pi)$ such that $K(\alpha)$ is hyperbolic for $\alpha \in [0, \alpha_K)$, Euclidean for $\alpha =\alpha_K$, and spherical for $\alpha \in (\alpha_K, 2\pi - \alpha_K)$. 
 
\subsubsection{Hyperbolic case} For $\alpha \in [0, \alpha_K)$, by the Schl\"{a}fli formula we have 
$$
\frac{d\mathrm{Vol}(K(\alpha))}{d\alpha}  = - \frac{1}{2} l_\alpha
$$
where $l_\alpha = \mathrm{Re}(\gamma_\alpha) > 0$ is the real length of $K \subset K(\alpha)$. Note that $K(\alpha)$ is Euclidean at $\alpha = \alpha_K$, so $\mathrm{Vol}(K(\alpha)) \to 0$ as $\alpha \to \alpha_K$. Let 
\begin{equation} \label{F}
F(\alpha)= i \int_{\overline{y_0}}^{y_0} \log \left( \frac{f^2_n(y)+ A^2}{(1+A^2)g_n(y)} \right) \frac{f'_n(y)dy}{f^2_n(y)-1}.
\end{equation} 
Then Theorem \ref{main-h} is equivalent to $\mathrm{Vol}(K(\alpha)) = F(\alpha)$. 

We first claim that $F(\alpha) \to 0$ as $\alpha \to \alpha_K$. Indeed, as $\alpha \to \alpha_K$ we have $l_\alpha \to 0$ and so $\gamma_\alpha = \ell_\alpha + i \varphi_\alpha \to i \varphi_{\alpha_K}$. Then, by the trigonometric identity (Proposition \ref{trig})  we obtain 
\begin{eqnarray*}
f_n(y_0) &=&  i \coth \left( \frac{\gamma_\alpha +  4n i  \alpha}{4} \right) \cot \left( \frac{\alpha}{2} \right) \\
&\to&   i \coth\left( \frac{i \varphi_{\alpha_K} +  4n i \alpha_K}{4} \right) \cot \left( \frac{\alpha_K}{2} \right) \\
&=&  \cot \left( \frac{\varphi_{\alpha_K} +  4n \alpha_K}{4} \right) \cot \left( \frac{\alpha_K}{2} \right),
\end{eqnarray*}
where we used $\coth (iz) = -i \cot(z)$.  Then $\mathrm{Im} \, f_n(y_0) \to 0$. Hence $f_n(\overline{y_0}) - f_n(y_0) = \overline{ f_n(y_0)} - f_n(y_0) = - 2 i \, \mathrm{Im} \, f_n(y_0) \to 0$. 

For $\alpha \in (\alpha_K - \ve, \alpha_K)$, with $\ve$ a sufficiently small positive real number, we let $s := f_n(y)$. Since $f_n(y)$ is a rational function $y$, we can write $y = h(s)$ for some continuous function $h(s)$ in a small open neighborhood of $f_n(y_0)$. Then, by changing variable we have
$$
F(\alpha)= i \int_{f_n(\overline{y_0})}^{f_n(y_0)} \log \left( \frac{s^2+ A^2}{(1+A^2) (g_n \circ h)(s)} \right) \frac{ds}{s^2-1}.
$$
As $\alpha \to \alpha_K$, since $f_n(\overline{y_0}) - f_n(y_0) \to 0$ we obtain $F(\alpha) \to 0$. 

Note that we also have $\mathrm{Vol}(K(\alpha)) \to 0$ as $\alpha \to \alpha_K$. Hence $\mathrm{Vol}(K(\alpha)) = F(\alpha)$  if we can show that 
$$
\frac{dF(\alpha)}{d\alpha} = \frac{d\mathrm{Vol}(K(\alpha))}{d\alpha} = - \frac{1}{2} l_\alpha.
$$. 

By taking derivative of \eqref{F} and noting that $dA/d\alpha=-(1+A^2)/2$, we have
\begin{eqnarray*}
 \frac{dF(\alpha)}{d\alpha}
&=&  \log \left( \frac{f^2_n(y_0)+ A^2}{(1+A^2)g_n(y_0)} \right) \frac{i f'_n(y_0)}{f^2_n(y_0)-1} \frac{dy_0}{d\alpha} - \log \left( \frac{f^2_n(\overline{y_0})+ A^2}{(1+A^2)g_n(\overline{y_0})} \right) \frac{i f'_n(\overline{y_0})}{f^2_n(\overline{y_0})-1} \frac{d\overline{y_0}}{d\alpha} 
 \\
&& + \, i \int_{\overline{y_0}}^{y_0} \frac{\partial}{\partial A} \left( \frac{f'_n(y)}{f^2_n(y)-1} \log \left( \frac{f^2_n(y)+ A^2}{(1+A^2)g_n(y)} \right) \right) \frac{dA}{d\alpha} \, dy
 \\
&=& i \int_{\overline{y_0}}^{y_0}\frac{f'_n(y)}{f^2_n(y)-1} \left( \frac{2A}{f^2_n(y)+ A^2} - \frac{2A}{1+ A^2} \right) \frac{-(1+A^2)}{2} dy \\
&=& i \int_{\overline{y_0}}^{y_0} \frac{f'_n(y) A}{f^2_n(y)+ A^2} dy \\
&=&i \left( \mathrm{arccot} \frac{f_n(\overline{y_0})}{A} - \mathrm{arccot} \frac{f_n(y_0)}{A}\right).
\end{eqnarray*}

Since
$
f_n(y_0) = i \coth \left( \frac{\gamma_\alpha +  4n i  \alpha}{4} \right) \cot \left( \frac{\alpha}{2} \right)
$ we have $\frac{f_n(y_0)}{A} = i \coth \left( \frac{\gamma_\alpha +  4n i  \alpha}{4} \right) = \cot \left( \frac{\gamma_\alpha +  4n i  \alpha}{4i} \right)$ and $\frac{f_n(\overline{y_0})}{A} =  \overline{\frac{f_n(y_0)}{A} } = \cot \left( \frac{\overline{\gamma_\alpha +  4n i  \alpha}}{-4i} \right)$. 
Hence 
\begin{eqnarray*}
\frac{dF(\alpha)}{d\alpha}  &=& i \left( \mathrm{arccot} \frac{f_n(\overline{y_0})}{A} - \mathrm{arccot} \frac{f_n(y_0)}{A}\right) \\
&=& i
 \left( \frac{\overline{\gamma_\alpha +  4n i  \alpha} }{-4i} - \frac{\gamma_\alpha +  4n i  \alpha}{4i} \right) \\
&=& -\frac{\overline{\gamma_\alpha} + \gamma_\alpha }{4} = - \frac{l_\alpha}{2}.
\end{eqnarray*}
This proves Theorem \ref{main-h} for $C(2n, 3)$ in the hyperbolic case.

\subsubsection{Spherical case} Equality \eqref{>pi} follows from Proposition \ref{pi}. For $\alpha \in (\alpha_K, \pi]$, by the Schl\"{a}fli formula we have 
$$
\frac{d\mathrm{Vol}(K(\alpha))}{d\alpha}  = \frac{1}{2} l_\alpha.
$$ 

Let 
\begin{equation} \label{G}
G(\alpha)= \int_{y_+}^{y_-} \log \left( \frac{f^2_n(y)+ A^2}{(1+A^2)g_n(y)} \right) \frac{f'_n(y)dy}{f^2_n(y)-1}.
\end{equation}
Then Theorem \ref{main-s} is equivalent to $\mathrm{Vol}(K(\alpha)) = G(\alpha)$. 

We first claim that $G(\alpha) \to 0$ as $\alpha \to \alpha_K$. Indeed,  as $\alpha \to \alpha_K$, we have $l_\alpha \to 0$ and 
$$
f_n(y_{\pm}) =  \cot \left( \frac{\varphi_\alpha \pm l_\alpha+  4n  \alpha}{4} \right) \cot \left( \frac{\alpha}{2} \right) \to  \cot \left( \frac{\varphi_{\alpha_K} +  4n \alpha_K}{4} \right) \cot \left( \frac{\alpha_K}{2} \right).
$$

For $\alpha \in (\alpha_K - \ve, \alpha_K)$, with $\ve$ a sufficiently small positive real number, we let $s := f_n(y)$. Since $f_n(y)$ is a rational function $y$, we can write $y = h(s)$ for some continuous function $h(s)$ in a small open neighborhood of $f_n(y_{\pm})$. Then, by changing variable we have
$$
G(\alpha)= \int_{f_n(y_+)}^{f_n(y_-)} \log \left( \frac{s^2+ A^2}{(1+A^2) (g_n \circ h)(s)} \right) \frac{ds}{s^2-1}.
$$
As $\alpha \to \alpha_K$, since $f_n(y_{\pm}) \to \cot \left( \frac{\varphi_{\alpha_K} +  4n \alpha_K}{4} \right) \cot \left( \frac{\alpha_K}{2} \right)$, we obtain $G(\alpha) \to 0$. Note that $\mathrm{Vol}(K(\alpha)) \to 0$ as $\alpha \to \alpha_K$. Hence $\mathrm{Vol}(K(\alpha)) = G(\alpha)$  if we can show that 
$$
\frac{dG(\alpha)}{d\alpha} = \frac{d\mathrm{Vol}(K(\alpha))}{d\alpha}  =  \frac{1}{2} l_\alpha.
$$ 

By taking derivative of \eqref{G} and noting that $dA/d\alpha=-(1+A^2)/2$, we have
\begin{eqnarray*}
 \frac{dG(\alpha)}{d\alpha}
&=&  \log \left( \frac{f^2_n(y_-)+ A^2}{(1+A^2)g_n(y_-)} \right) \frac{f'_n(y_-)}{f^2_n(y_-)-1} \frac{dy_-}{d\alpha} - \log \left( \frac{f^2_n(y_+)+ A^2}{(1+A^2)g_n(y_+)} \right) \frac{f'_n(y_+)}{f^2_n(y_+)-1} \frac{dy_+}{d\alpha} 
 \\
&& + \,  \int_{y_+}^{y_-} \frac{\partial}{\partial A} \left( \frac{f'_n(y)}{f^2_n(y)-1} \log \left( \frac{f^2_n(y)+ A^2}{(1+A^2)g_n(y)} \right) \right) \frac{dA}{d\alpha} \, dy
 \\
&=&  \int_{y_+}^{y_-}\frac{f'_n(y)}{f^2_n(y)-1} \left( \frac{2A}{f^2_n(y)+ A^2} - \frac{2A}{1+ A^2} \right) \frac{-(1+A^2)}{2} dy \\
&=&  \int_{y_+}^{y_-} \frac{f'_n(y) A}{f^2_n(y)+ A^2} dy \\
&=& \mathrm{arccot} \frac{f_n(y_+)}{A} -  \mathrm{arccot} \frac{f_n(y_-)}{A}.
\end{eqnarray*}

By the trigonometric identity (Proposition \ref{trig})  we have
$f_n(y_{\pm}) =  \cot \left( \frac{\varphi_\alpha \pm l_\alpha+  4n  \alpha}{4} \right) \cot \left( \frac{\alpha}{2} \right)$.  This implies that $\frac{f_n(y_{\pm})}{A} =  \cot \left( \frac{\varphi_\alpha \pm l_\alpha+  4n  \alpha}{4} \right)$. 
Hence 
\begin{eqnarray*}
\frac{dG(\alpha)}{d\alpha}  &=& - \mathrm{arccot} \frac{f_n(y_-)}{A} +  \mathrm{arccot} \frac{f_n(y_+)}{A}  \\
&=& - \frac{\varphi_\alpha - l_\alpha+  4n  \alpha}{4} +\frac{\varphi_\alpha + l_\alpha+  4n  \alpha}{4} \\
&=& \frac{l_\alpha}{2}.
\end{eqnarray*}
This proves Theorem \ref{main-s} for $C(2n, 3)$. 
\section{$C(2n, 2p)$}
\label{even}

\subsection{Knot group}

\begin{proposition}
We have $G(C(2n, 2p)) = \la a, b \mid \omega' a = b \omega' \ra$ where 
$$\omega' = [(a^{-1}b)^{n}(ab^{-1})^n]^p.$$
\end{proposition}

\begin{figure}[h]
		\centering
		\includegraphics[scale=1]{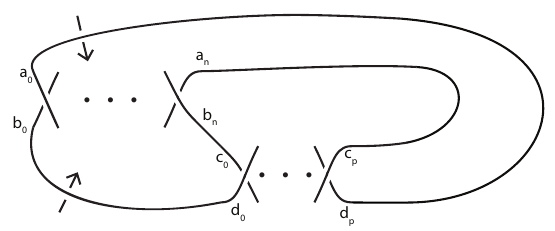}
		\caption{$C(2n,2p)$}
\end{figure}

\begin{proof}
The proof is similar to that of  Proposition \ref{group-odd} for $C(2n, 2p + 1)$, with
appropriate modifications for the even case. Starting from the knot diagram and using the Wirtinger presentation, we trace through the crossings to obtain the stated relation.
\end{proof}

Note that $C(2n, 2p)$ is the double twist knot $J(-2n, 2p)$, so a similar presentation for its knot group can also be found in \cite{HS, MPL}.

\subsection{$\mathrm{SL}_2(\BC)$-representations}

Suppose $\rho\colon G(C(2n, 2p))\to  \mathrm{SL}_2(\BC)$ is a nonabelian representation. 
Up to conjugation, we may assume that 
$$
A:=\rho(a) = \left[
\begin{array}{cc}
m & 1\\
0 & m^{-1}
\end{array}
\right] \quad \text{and} \quad B:=\rho(b) = \left[
\begin{array}{cc}
m & 0\\
2-z & m^{-1}
\end{array}
\right]
$$
where $(m,z) \in \BC^2$ satisfies $\rho(\omega' a) = \rho(b \omega')$. Note that $z=\tr \rho(ab^{-1})$. 

The following propositions are proved in \cite{MPL, MT}. 

\begin{proposition} \label{v}
	Let  $V = \rho((a^{-1} b)^{n} (ab^{-1})^{n})$ and $v = \tr V$. Then
	\[
		v = 2 + (z-2) (z+2-x^2) S^2_{n-1}(z).
	\]
\end{proposition}

\begin{proposition} \label{R'}
We have 
$$
\rho(\omega' a)  -  \rho(b \omega') =  
\left[
\begin{array}{cc}
0 & \Phi_{C(2n, 2p)}(x,z)\\
(z-2) \Phi_{C(2n, 2p)}(x,z) & 0
\end{array}
\right]
$$
where 
$$ 
\Phi_{C(2n, 2p)}(x,z) :=  \left[ 1 + (z+2-x^2) S_{n-1}(z) (S_{n}(z)  - S_{n-1}(z)) \right] S_{p-1}(v) - S_{p-2}(v).
$$
\end{proposition}

\subsection{Longitude and trignometric identity}

If we choose the meridian to be $\mu=b$ then the canonical longitude is $\lambda = \omega' (\omega')^*$, where $(\omega')^*$ is the word obtained from $\omega'$ by writing the letters in $\omega'$ in reversed order. Note that since $\omega' a = b \omega'$ and $(\omega')^* b = a (\omega')^*$ we have $\omega' (\omega')^* b = \omega' a (\omega')^* = b \omega' (\omega')^*$. 

Since $\rho(\mu) = \left[
\begin{array}{cc}
m & 0\\
* & m^{-1}
\end{array}
\right]$ we have $\rho(\lambda) = \left[
\begin{array}{cc}
l & 0\\
* & l^{-1}
\end{array}
\right]$, where $m = e^{i \alpha/2}$ and $l = e^{\gamma_\alpha/2}$. By a similar calculation as in \cite[Equation (2.3)]{HS} we have $l = - W'_{12} \big/ \widetilde{W'_{12}}$, where $W'_{12}$ is the $(1,2)$-entry of $W'=\rho(\omega')$ and $\widetilde{W'_{12}}$ is obtained from $W'_{12}$ by replacing $m$ by $m^{-1}$. Note that $W'_{12}$ is a function in $m$ and $z$. 

Similar to Propositions \ref{12} and \ref{fn} we have the following propositions. 

\begin{proposition} \label{12'}
We have 
$$
W'_{12} =  \left( m (S_{n}(z) - S_{n-1}(z)) - m^{-1} (S_{n-1}(z) - S_{n-2}(z) \right) S_{n-1}(z) S_{p-1}(v). 
$$
\end{proposition}

\begin{proposition}\label{trig'}
Let $f_n(z) = \frac{zS_{n-1}(z) - 2S_n(z)}{(z-2)S_{n-1}(z)}$. Then if $\Phi_{C(2n, 2p)}(x,z) =0$, where $x=m+m^{-1}$, we have  
$$
 f_n(z) = - \frac{l^{1/2} +l^{-1/2}}{l^{1/2}-l^{-1/2}} \cdot \frac{m+m^{-1}}{m-m^{-1}}.
$$ 
\end{proposition}

\subsection{Proof of Theorems \ref{main-h} and \ref{main-s} for $C(2n,2)$}

Suppose $p=1$. By Propositions \ref{v} and \ref{R'} we have  $v = 2 + (z-2) (z+2-x^2) S^2_{n-1}(z)$ and 
\begin{eqnarray*}
\Phi_{C(2n, 2)}(x,z) &=& 1 + (z+2-x^2) S_{n-1}(z) (S_{n}(z)  - S_{n-1}(z)).
\end{eqnarray*} 

We write $\Phi_{C(2n,2)}(x,z) = b - a x^2$, where 
\begin{eqnarray*}
a &=& S_{n-1}(z) (S_{n}(z)  - S_{n-1}(z)), \\
b &=& 1 + (z+2) S_{n-1}(z) (S_{n}(z)  - S_{n-1}(z)).
\end{eqnarray*} 
Since $S^2_n(z) - z S_n(z) S_{n-1}(z) +  S^2_{n-1}(z) = 1$ we have
\begin{eqnarray*}
b - 4a &=& 1 + (z-2) S_{n-1}(z) (S_{n}(z)  - S_{n-1}(z)) \\
&=& S^2_{n}(z)  - 2 S_{n}(z) S_{n-1}(z) + (3-z)S^2_{n-1}(z).
\end{eqnarray*} 

Choose $c = f_n(z)$. By Lemma \ref{cd}, the equation $\Phi_{C(2n, 2)}(2\cos \frac{\alpha}{2}, z)=0$ is equivalent to $c^2+ A^2= (1+A^2)d$ where $d = 1 + (c^2-1)(1 - \frac{b}{4a})$. Since 
$$
c^2 -1 = \frac{4(S_n(z)-S_{n-1}(z)) (S_n(z) - (z-1) S_{n-1}(z))}{(z-2)^2 S^2_{n-1}(z)},
$$ 
by a direct calculation we have
\begin{eqnarray*}
d &=& 1 + (c^2-1) \left( 1 - \frac{b}{4a} \right) \\
&=& 1 -  \frac{\left( S_n(z) - (z-1) S_{n-1}(z) \right) \left( S^2_{n}(z)  - 2 S_{n}(z) S_{n-1}(z) + (3-z)S^2_{n-1}(z) \right)}{(z-2)^2 S^3_{n-1}(z)} \\
&=& - \frac{\left( S_n(z) -S_{n-1}(z) \right)\left( S^2_n(z) + S^2_{n-1}(z) - z S_n(z) S_{n-1}(z) \right)}{(z-2)^2 S^3_{n-1}(z)} \\
&=& - \frac{ S_n(z) -S_{n-1}(z)}{(z-2)^2 S^3_{n-1}(z)}.
\end{eqnarray*} 

Hence we have proved the following. 

\begin{proposition} \label{Riley}
Let 
$$
f_n(z) = \frac{zS_{n-1}(z) - 2S_n(z)}{(z-2)S_{n-1}(z)}, \qquad g_n(z) = - \frac{ S_n(z) -S_{n-1}(z)}{(z-2)^2 S^3_{n-1}(z)}.
$$
Let $x = 2\cos \frac{\alpha}{2}$ and $A = \cot \frac{\alpha}{2}$. Then the equation $\Phi_{C(2n, 2)}(x,z) = 0$ is equivalent to $f^2_n(z)+ A^2= (1+A^2)g_n(z)$.
\end{proposition}

By using the trigonometric identity (Proposition \ref{trig'}) and Proposition \ref{Riley}, the proof of Theorems \ref{main-h} and \ref{main-s} for $C(2n,2)$  is similar to that for $C(2n,3)$. 

\subsection{Proof of Theorems \ref{main-h} and \ref{main-s} for $C(2n, -2n)$} Suppose $p=-2n$. Note that $C(2n, -2n)$ is the mirror image of the double twist knot $J(2n, 2n)$ in \cite{HS}. By \cite{MPL} the component of $\Phi_{C(2n, -2n)}(x,z)$ containing  the holonomy representation is a factor of $v - z = (z-2) \left( -1 + (z+2-x^2) S^2_{n-1}(z) \right)$. The factor $z-2$ corresponds to reducible representations, hence the factor 
$$\Phi^{\mathrm{hol}}_{C(2n, -2n)}(x,z) := -1 + (z+2-x^2) S^2_{n-1}(z)$$ determines the component containing the holonomy representation.

We write $\Phi^{\mathrm{hol}}_{C(2n, -2n)}(x,z) = b - a x^2$, where $a = S^2_{n-1}(z)$ and $ b = -1 + (z+2) S^2_{n-1}(z)$. Since $S^2_n(z) - z S_n(z) S_{n-1}(z) +  S^2_{n-1}(z) = 1$ we have
$$
b - 4a = - 1 +(z-2) S^2_{n-1}(z)
= - S^2_{n}(z)  +  z S_{n}(z) S_{n-1}(z) + (z-3)S^2_{n-1}(z).
$$

Choose $c = f_n(z)$. By Lemma \ref{cd}, the equation $\Phi^{\mathrm{hol}}_{C(2n, -2n)}(2\cos \frac{\alpha}{2}, z)=0$ is equivalent to $c^2+ A^2= (1+A^2)d$ where $d = 1 + (c^2-1)(1 - \frac{b}{4a})$. Since 
$$
c^2 -1 = \frac{4(S_n(z)-S_{n-1}(z)) (S_n(z) - (z-1) S_{n-1}(z))}{(z-2)^2 S^2_{n-1}(z)},
$$ 
by a direct calculation we have
\begin{eqnarray*}
d &=& 1 + (c^2-1) \left( 1 - \frac{b}{4a} \right) \\
&=& 1 -  \frac{(S_n(z)-S_{n-1}(z)) (S_n(z) - (z-1) S_{n-1}(z)) }{(z-2)^2 S^4_{n-1}(z)} \\
&& \qquad \times \left( - S^2_{n}(z)  +  z S_{n}(z) S_{n-1}(z) + (z-3)S_{n-1}(z) \right) \\
&=& \frac{\left( S^2_n(z) + S^2_{n-1}(z) - z S_n(z) S_{n-1}(z) \right)^2}{(z-2)^2 S^4_{n-1}(z)} \\
&=& \frac{1}{(z-2)^2 S^4_{n-1}(z)}.
\end{eqnarray*} 

Hence we have proved the following. 

\begin{proposition} \label{Riley'}
Let 
$$
f_n(z) = \frac{zS_{n-1}(z) - 2S_n(z)}{(z-2)S_{n-1}(z)}, \qquad g_n(z) = \frac{1}{(z-2)^2S^4_{n-1}(z)}.
$$
Let $x = 2\cos \frac{\alpha}{2}$ and $A = \cot \frac{\alpha}{2}$. Then the equation $\Phi^{\mathrm{hol}}_{C(2n, -2n)}(x,z) = 0$ is equivalent to $f^2_n(z)+ A^2= (1+A^2)g_n(z)$.
\end{proposition}

By using the trigonometric identity (Proposition \ref{trig'}) and Proposition \ref{Riley'}, the proof of Theorems \ref{main-h} and \ref{main-s} for $C(2n,-2n)$  is similar to that for $C(2n,3)$.

\section{Examples} \label{ex}

We illustrate the volume formulas from Theorem~\ref{main-h} and
Theorem~\ref{main-s} with explicit computations for representative knots from
each family.

For a two-bridge knot $K=C(2n,3)$, $C(2n, 2)$ or $C(2n, -2n)$ with cone
angle $\alpha\in[0,\alpha_K)$, Theorem~\ref{main-h} expresses the hyperbolic
volume as
\[
  \mathrm{Vol}(K(\alpha))
  = i\!\int_{\bar{y}_0}^{y_0}
    \CI(y)\,dy,
\]
where $y_0$ is a root of the basic equation $f^2_n(y)+ A^2= (1+A^2)g_n(y)$ with $\mathrm{Im}(f_n(y_0))>0$,
$\bar{y}_0$ is its conjugate, and the integrand is
$$
  \CI(y)
  \;=\;
  \log\!\left(\frac{f_n(y)^2 + A^2}{(1+A^2)\,g_n(y)}\right)
  \frac{f_n'(y)}{f_n(y)^2-1},
  \qquad
  A = \cot\!\tfrac{\alpha}{2}.
$$
When $\alpha\in(\alpha_K,\pi]$ (spherical regime) and the basic equation
$f_n^2+A^2=(1+A^2)g_n$ has roots $y_\pm$ with $f_n(y_\pm)\in\BR$,
Theorem~\ref{main-s} gives the volume as the integral
$\int_{y_+}^{y_-}\CI(y)\,dy$.

In all cases, the integration path must not pass through singular points
of the integrand $\CI(y)$. When such a singularity
lies on the straight-line path from $\bar{y}_0$ to $y_0$ (in the hyperbolic case) or from $y_+$ to $y_-$ (in the spherical case), the contour
is deformed to avoid it, as noted in the statements of Theorem~\ref{main-h} and Theorem~\ref{main-s}.

\subsection{Example 1: Figure-eight knot $4_1 = C(2,2)$}

The rational slope is $5/2 = 2+1/2$, so $4_1 = \fb(5,2)$ in the Schubert notation. 

Since $f_1(y) = -\dfrac{y}{y-2}$ and $g_1(y) = -\dfrac{y-1}{(y-2)^2}$
(Proposition \ref{Riley}), the basic equation $f_1^2+A^2=(1+A^2)g_1$  simplifies to
\[
  y^2-(1+2\cos\alpha)\,(y-1)=0.
\]

From the computation
\[
  \frac{f_1'(y)}{f_1(y)^2-1} = \frac{1}{2(y-1)},
  \qquad
  \frac{f_1(y)^2 + A^2}{(1+A^2)\,g_1(y)}
  = -\frac{y^2 + A^2(y-2)^2}{(1+A^2)(y-1)},
\]
the integrand for $4_1$  is
$$
  \CI(y)
  = \log\!\left(-\frac{y^2+A^2(y-2)^2}{(1+A^2)(y-1)}\right)
    \frac{1}{2(y-1)}.
$$

\textbf{Euclidean angle:} $\alpha_K=2\pi/3$. 

\textbf{Case $\alpha=\pi$:} $A=\cot(\alpha/2) =0$. Setting $\cos\alpha = -1$, the basic equation becomes $y^2+y-1=0$, with roots
$y=\frac{-1\pm\sqrt{5}}{2}$. Then
\[
  \mathrm{Vol}(4_1(\pi)) = \int_{\frac{-1-\sqrt{5}}{2}}^{\frac{-1+\sqrt{5}}{2}} \log\!\left(-\frac{y^2}{y-1}\right)
    \frac{1}{2(y-1)} dy= 1.97392 \ldots = \frac{\pi^2}{5}.
\]

\textbf{Lobachevsky verification.}
Recall that $\Lambda(\theta)=-\int_0^\theta\ln|2\sin t|\,dt$.
We verify Theorem~\ref{main-h} against the Lobachevsky closed forms in \cite{MV}
for the following three hyperbolic cone angles.

\smallskip
\textbf{Case $\alpha=\pi/2$:} $A=\cot(\alpha/2) =1$. 
Setting $\cos\alpha=0$, the basic equation becomes $y^2-y+1=0$, with roots
$y = \tfrac{1 \pm i\sqrt{3}}{2}$.
Then 
\begin{eqnarray*}
  \mathrm{Vol}(4_1(\pi/2))
  &=& i\!\int_{\tfrac{1 - i\sqrt{3}}{2}}^{\tfrac{1 + i\sqrt{3}}{2}}\log\!\left(-\frac{y^2+(y-2)^2}{2(y-1)}\right)
    \frac{1}{2(y-1)}dy \\
  &=& 0.50747\ldots
  = \tfrac{3}{2}\,\Lambda(\pi/3)
  \quad\text{\cite{MV}}.
\end{eqnarray*} 

\textbf{Case $\alpha=\pi/3$:} $A=\cot(\alpha/2) =\sqrt{3}$.
Setting $\cos\alpha=\tfrac{1}{2}$, the basic equation becomes $y^2-2y+2=0$,
with roots $y = 1 \pm i$.  Then
\begin{eqnarray*}
  \mathrm{Vol}(4_1(\pi/3))
    &=&  i\!\int_{1 - i}^{1+i}\log\!\left(-\frac{y^2+3(y-2)^2}{4(y-1)}\right)
    \frac{1}{2(y-1)}dy \\
  &=& 1.22129\ldots
  = \tfrac{8}{3}\,\Lambda(\pi/4)
  \quad\text{\cite{MV}}.
\end{eqnarray*} 
Note that the straight-line path from $1-i$ to $1+i$ passes through the singularity $y=1$ of the integrand, so we can choose the contour consisting of  the  straight-line path from $1-i$ to $0$ and the straight-line path from $0$ to $1+i$.

\textbf{Case $\alpha=\pi/5$:} $A=\cot(\alpha/2) =\sqrt{5+ 2\sqrt{5}}$.
Setting $\cos\alpha=\cos(\pi/5) = \tfrac{1+\sqrt{5}}{4}$, the basic equation
becomes $y^2 - \tfrac{3+\sqrt{5}}{2}\,y + \tfrac{3+\sqrt{5}}{2} = 0$,
with roots
$y = (3+\sqrt{5} \pm \sqrt{10+2\sqrt{5}}\,i)/4$. Then
\begin{eqnarray*}
  \mathrm{Vol}(4_1(\pi/5))
  &=& i\!\int_{(3+\sqrt{5} + \sqrt{10+2\sqrt{5}}\,i)/4}^{(3+\sqrt{5} - \sqrt{10+2\sqrt{5}}\,i)/4} \log\!\left(-\frac{y^2+(5+ 2\sqrt{5})(y-2)^2}{(6+ 2\sqrt{5})(y-1)}\right)
    \frac{1}{2(y-1)} dy \\
 &=& 1.70857\ldots
  = 2\Lambda(3\pi/10)+2\Lambda(\pi/10)
  \qquad\text{\cite{MV}}.
\end{eqnarray*}

All above values are independently verified with SnapPy~\cite{Sn}.

\subsection{Example 2: Knot $5_2 = C(2,3)$}

The rational slope is $7/3 = 2+1/3$, so $5_2 = \fb(7,3)$.
Since $f_1(y)=-\dfrac{y}{y-2}$ and $g_1(y)=-\dfrac{(y-1)^2}{(y-2)^3}$ (Proposition \ref{Riley-odd}), 
the basic equation $f_1^2+A^2=(1+A^2)g_1$  simplifies to 
$$
  y^3 - (3+2\cos\alpha)\,y^2 + (4+6\cos\alpha)\,y - (3+4\cos\alpha) = 0.
$$

From the computations
\[
  \frac{f_1'(y)}{f_1(y)^2-1} = \frac{1}{2(y-1)},
  \qquad
  \frac{f_1(y)^2 + A^2}{(1+A^2)\,g_1(y)}
  = -(y-2)\frac{y^2 + A^2(y-2)^2}{(1+A^2)(y-1)^2},
\]
the integrand for $5_2$  is
$$
  \CI(y)
  = \log\!\left(-(y-2)\frac{y^2 + A^2(y-2)^2}{(1+A^2)(y-1)^2}\right)
    \frac{1}{2(y-1)}.
$$

\textbf{Case $\alpha=\pi$:} $A=\cot(\alpha/2) =0$. 
Setting $\cos\alpha=-1$, the basic equation becomes $y^3-y^2-2y+1=0$,
with roots  $\approx -1.24698, \, 0.44504, \, 1.80194$. With $y_+ \approx -1.24698$ and $y_- \approx 0.44504$ we obtain 
\[
  \mathrm{Vol}(5_2(\pi))= \int_{y_+}^{y_-} \log\!\left(-(y-2)\frac{y^2}{(y-1)^2}\right)
    \frac{1}{2(y-1)} dy= 1.40994 \ldots = \frac{\pi^2}{7}.
\]

\textbf{Case $\alpha=2\pi/3$:} $A=\cot(\alpha/2) =1/\sqrt{3}$. 
Setting $\cos\alpha = -\tfrac{1}{2}$, the basic equation becomes
$y^3-2y^2+y-1=0$, which has  a real root $\approx 1.75488$ and two complex roots $\approx 0.12256 \pm 0.74486 \, i$. With $y_0 \approx 0.12256 + 0.74486 \, i$ we obtain
\[
  \mathrm{Vol}(5_2(2\pi/3))
  = i\!\int_{\overline{y_0}}^{y_0}\ \log\!\left(-(y-2)\frac{3y^2 + (y-2)^2}{4(y-1)^2}\right)
    \frac{1}{2(y-1)} \,dy
  = 0.31424\ldots
\]
This agrees with \cite[Example~3]{Me}.

\textbf{Case $\alpha=\pi/2$:} $A=\cot(\alpha/2) =1$.
Setting $\cos\alpha = 0$, the basic equation becomes
$y^3-3y^2+4y-3=0$, which has a real root $\approx 1.68233$ and two complex conjugate roots $\approx 0.65884 \pm 1.16154 \, i$. With $y_0 \approx 0.65884 + 1.16154 \, i$ we obtain
\[
  \mathrm{Vol}(5_2(\pi/2))
  =  i\!\int_{\overline{y_0}}^{y_0}\ \log\!\left(-(y-2)\frac{y^2 + (y-2)^2}{2(y-1)^2}\right)
    \frac{1}{2(y-1)} \,dy
  = 1.18738\ldots
\]

\textbf{Case $\alpha=\pi/3$:} $A=\cot(\alpha/2) =\sqrt{3}$.
Setting $\cos\alpha = \tfrac{1}{2}$, the equation becomes
$y^3-4y^2+7y-5=0$,
which has a real root $\approx 1.56984$ and two complex conjugate roots $\approx 1.21508 \pm 1.30714 \, i$. With $y_0 \approx 1.21508 \pm 1.30714 \, i$ we obtain
\[
  \mathrm{Vol}(5_2(\pi/3))
  =  i\!\int_{\overline{y_0}}^{y_0}\ \log\!\left(-(y-2)\frac{y^2 + 3(y-2)^2}{4(y-1)^2}\right)
    \frac{1}{2(y-1)} \,dy
  = 2.04253\ldots
\]

All above values are independently verified with SnapPy~\cite{Sn}.

\subsection{Example 3: Stevedore's knot $6_1 = C(4,2)$}

The rational slope is $9/2 = 4+1/2$, so $6_1 = \fb(9,2)$.
Since $f_2(y) =- \dfrac{y^2-2}{(y-2)\,y}$ and 
$g_2(y) = -\dfrac{y^2-y-1}{(y-2)^2\,y^3}$ (Proposition \ref{Riley}), 
the basic equation $f_2^2+A^2=(1+A^2)g_2$  simplifies to 
$$
  y^5 - 2(1{+}\cos\alpha)\,y^4 + (4\cos\alpha) y^3
  + y^2 + (1{-}2\cos\alpha)\,y - 1 = 0.
$$

From the computations
\[
  \frac{f_2'(y)}{f_2(y)^2-1}
  = \frac{y^2-2y+2}{2(y-1)(y^2-y-1)},
  \qquad
  \frac{f_2(y)^2+A^2}{(1+A^2)\,g_2(y)}
  = -y \frac{(y^2{-}2)^2+A^2(y{-}2)^2y^2}{(1+A^2)(y^2{-}y{-}1)},
\]
the integrand for $6_1$  is
$$
  \CI(y)
  = \log\!\left(-y \frac{(y^2{-}2)^2+A^2(y{-}2)^2y^2}{(1+A^2)(y^2{-}y{-}1)}\right)
    \frac{y^2-2y+2}{2(y-1)(y^2-y-1)}.$$

\textbf{Case $\alpha=\pi$:} $A=\cot(\alpha/2) =0$. 
Setting $\cos\alpha=-1$, the basic equation becomes $y^5  -4y^3
  + y^2 + 3y - 1=0$,
with roots  $\approx -1.87939, \, -1, \, 0.34730, \, 1, \, 1.53209$. With $y_+ \approx -1.87939$ and $y_-=-1$ we obtain
\[
  \mathrm{Vol}(6_1(\pi))= \int_{y_+}^{y_-}\log\!\left(-y \frac{(y^2{-}2)^2}{y^2{-}y{-}1}\right)
    \frac{y^2-2y+2}{2(y-1)(y^2-y-1)} dy= 1.09662 \ldots = \frac{\pi^2}{9}.
\]

\textbf{Case $\alpha = 2\pi/3$:} $A=\cot(\alpha/2) =1/\sqrt{3}$. 
Setting $\cos\alpha = -\tfrac{1}{2}$, the basic equation becomes 
$y^5 - y^4 -2 y^3+ y^2 + 2y - 1 = 0$, which has three real roots $\approx 0.52489, 1, 1.49022$ and two complex roots $\approx -1.00755 \pm 0.51312 \,i$. With $y_0 \approx -1.00755 + 0.51312 \,i$ we obtain 
$$
  \mathrm{Vol}(6_1(2\pi/3))
  = i\!\int_{\overline{y_0}}^{y_0} \log\!\left(-y \frac{3(y^2{-}2)^2+(y{-}2)^2y^2}{4(y^2{-}y{-}1)}\right)
    \frac{y^2-2y+2}{2(y-1)(y^2-y-1)} dy \\
  = 0.65425\ldots
$$

\textbf{Case $\alpha = \pi/2$:} $A=\cot(\alpha/2) =1$.
Setting $\cos\alpha = 0$, the basic equation becomes $y^5 - 2 y^4 + y^2 + y - 1 =(y-1)^2(y^3-y-1) = 0$, which has three real roots $1 \text{~(multiplicity two)}, \, 1.32472$ and two complex roots $\approx -0.66236 \pm 0.56228 \, i$. With $y_0 \approx -0.66236 + 0.56228 \, i$ we obtain       
$$
  \mathrm{Vol}(6_1(\pi/2))
  = i\!\int_{\overline{y_0}}^{y_0}  \log\!\left(-y \frac{(y^2{-}2)^2+(y{-}2)^2y^2}{2(y^2{-}y{-}1)}\right)
    \frac{y^2-2y+2}{2(y-1)(y^2-y-1)} dy \\
  = 1.64974\ldots 
$$

\textbf{Case $\alpha = \pi/3$:} $A=\cot(\alpha/2) =\sqrt{3}$.
Setting $\cos\alpha = \tfrac{1}{2}$, the basic equation becomes
$y^5 - 3y^4 + 2 y^3 + y^2  - 1 = 0$, which has a real root $1$ and four complex roots $\approx -0.47356 \pm 0.44477 \, i, \, 1.47356 \pm 0.44477 \, i$. With $y_0 \approx -0.47356 + 0.44477 \, i$ we obtain
$$
  \mathrm{Vol}(6_1(\pi/3))
  = i\!\int_{\overline{y_0}}^{y_0}  \log\!\left(-y \frac{(y^2{-}2)^2+3(y{-}2)^2y^2}{4(y^2{-}y{-}1)}\right)
    \frac{y^2-2y+2}{2(y-1)(y^2-y-1)} dy \\
  = 2.47479\ldots 
$$
This agrees with \cite[Example~4]{Me}.

All above values are  independently verified with SnapPy~\cite{Sn}.

\subsection{Example 4: Knot $7_4 = C(4,-4)$}

The rational slope is $15/4 = 4 - 1/4$, so $7_4=\fb(15,4)$. Since 
$f_2(y) = -\dfrac{y^2-2}{(y-2)y}$ and $g_2(y) = \dfrac{1}{(y-2)^2y^4}$ (Proposition \ref{Riley'}), the basic equation $f_2^2+A^2=(1+A^2)g_2$ simplifies to 
$$
  1-y^3+2y^2\cos\alpha = 0.
$$

From the computations
\[
  \frac{f_2'(y)}{f_2(y)^2-1}
  = \frac{y^2-2y+2}{2(y-1)(y^2-y-1)},
  \qquad
  \frac{f_2(y)^2+A^2}{(1+A^2)\,g_2(y)}
  = y^2 \frac{(y^2{-}2)^2+A^2(y{-}2)^2y^2}{1+A^2},
\]
the integrand for $7_4$  is
$$
  \CI(y)
  = \log\!\left(y^2 \frac{(y^2{-}2)^2+A^2(y{-}2)^2y^2}{1+A^2}\right)
    \frac{y^2-2y+2}{2(y-1)(y^2-y-1)}.$$

\textbf{Case $\alpha=\pi$:} $A=\cot(\alpha/2) =0$.
Setting $\cos\alpha=-1$, the basic equation becomes $y^3
  + 2y^2  - 1=0$,
with roots  $\approx -1.61803, \, -1,  \, 0.61803$. With $y_+ \approx -1.61803$ and $y_-=-1$ we obtain
\[
  \mathrm{Vol}(7_4(\pi))= \int_{y_+}^{y_-}\log\!\left(y^2 (y^2{-}2)^2\right)
    \frac{y^2-2y+2}{2(y-1)(y^2-y-1)} dy= 0.65797 \ldots = \frac{\pi^2}{15}.
\]

\textbf{Case $\alpha = 2\pi/3$:} $A=\cot(\alpha/2) =1/\sqrt{3}$. 
Setting $\cos\alpha = -\tfrac{1}{2}$, the basic equation becomes 
$y^3+y^2-1= 0$, which has one real root $\approx 0.75488$ and two complex roots $\approx -0.87744 \pm 0.74486\,i$. With $y_0 \approx -0.87744 + 0.74486\,i$ we obtain 
$$
  \mathrm{Vol}(7_4(2\pi/3))
  = i\!\int_{\overline{y_0}}^{y_0} \log\!\left(y^2 \frac{3(y^2{-}2)^2+(y{-}2)^2y^2}{4}\right)
    \frac{y^2-2y+2}{2(y-1)(y^2-y-1)}dy \\
  = 1.57118\ldots
$$

\textbf{Case $\alpha = \pi/2$:} $A=\cot(\alpha/2) =1$.
Setting $\cos\alpha = 0$, the basic equation becomes $y^3-1= 0$, which has one real root $1$ and two complex roots $\frac{-1 \pm \sqrt{3} i}{2}$. With $y_0 = \frac{-1 + \sqrt{3} i}{2}$ we obtain 
$$
  \mathrm{Vol}(7_4(\pi/2))
  = i\!\int_{\overline{y_0}}^{y_0} \log\!\left(y^2 \frac{(y^2{-}2)^2+(y{-}2)^2y^2}{2}\right)
    \frac{y^2-2y+2}{2(y-1)(y^2-y-1)}dy \\
  = 3.04482\ldots
$$

\textbf{Case $\alpha = \pi/3$:} $A=\cot(\alpha/2) =\sqrt{3}$.
Setting $\cos\alpha = \tfrac{1}{2}$, the basic equation becomes
$y^3-y^2-1= 0$, which has a real root $\approx 1.46557$ and two complex roots $\approx -0.23279 \pm 0.79255 \, i$. With $y_0 \approx -0.23279 + 0.79255 \, i$ we obtain 
$$
  \mathrm{Vol}(7_4(\pi/3))
  = i\!\int_{\overline{y_0}}^{y_0} \log\!\left(y^2 \frac{(y^2{-}2)^2+3(y{-}2)^2y^2}{4}\right)
    \frac{y^2-2y+2}{2(y-1)(y^2-y-1)}dy \\
  = 4.22178\ldots
$$

All above values are  independently verified with SnapPy~\cite{Sn}.

\subsection{Summary of cone-manifold volumes}

The following table collects the values of
$\mathrm{Vol}(K(2\pi/m))$ computed above for $4_1$, $5_2$, $6_1$, $7_4$ and $m=2, 3, 4, 6$. All entries are computed via Theorems~\ref{main-h} and~\ref{main-s}, and 
are independently verified
with SnapPy~\cite{Sn}.

\medskip 

\begin{center}
\begin{tabular}{lclclclclcl}
\hline
Knot & $m=2$ & $m=3$ & $m=4$ & $m=6$ \\ 
\hline
$4_1$ & $\pi^2/5\approx 1.97392$ & 0 (Euclidean) & $0.50747\ldots$ & $1.22129\ldots$ \\ 
$5_2$ & $\pi^2/7\approx 1.40994$ & $0.31424\ldots$ & $1.18738\ldots$ & $2.04253\ldots$ \\ 
$6_1$ & $\pi^2/9\approx 1.09662$ & $0.65425\ldots$ & $1.64974\ldots$ & $2.47479\ldots$ \\ 
$7_4$ & $\pi^2/15\approx 0.65797$ & $1.57118\ldots$ & $3.04482\ldots$ & 4.22178\ldots \\ 
\hline
\end{tabular}
\end{center}

\medskip

As shown in Example~1, the Lobachevsky closed forms for $4_1$ at $m=4,6,10$
from Vesnin--Mednykh~\cite{MV} are reproduced approximately by our integral formula.
The values $\mathrm{Vol}(5_2(2\pi/3))=0.31424\ldots$ and
$\mathrm{Vol}(6_1(\pi/3))=2.47479\ldots$ agree with
\cite[Examples~3 and~4]{Me}, respectively.

\section*{Acknowledgements} The authors would like to thank the referees for helpful suggestions and comments. 
The first author has been supported by a grant from the Simons Foundation (\#708778). 

\section*{Declarations} 

\subsection*{Conflict of Interest} The authors declare no competing interests.

\subsection*{Data Availability} Data sharing is not applicable to this article as no datasets were generated or analyzed during the current study.


\begin{thebibliography}{99999}

\bibitem[HL]{HL} J. Ham and J. Lee, {\em The volume of hyperbolic cone-manifolds of the knot with Conway's notation $C(2n,3)$}, J. Knot Theory Ramifications \textbf{25} (2016) 1650030.

\bibitem[HMP]{HMP} J. Ham, A. Mednykh, and V. Petrov, {\em Identities and volumes of the hyperbolic twist knot cone-manifolds}, J. Knot Theory Ramifications \textbf 23 (2014) 1450064.

\bibitem[HLM]{HLM} H. Hilden, M. Lozano, and J. Montesinos-Amilibia, {\em Volumes and Chern-Simons invariants of cyclic coverings over rational knots},  in Topology and Teichm\"uller spaces (Katinkulta, 1995), pages 31--55. World Sci. Publ., River Edge, NJ, 1996.

\bibitem[HS]{HS} J. Hoste and P. Shanahan, {\em A formula for the A-polynomial of twist knots}, J. Knot Theory Ramifications \textbf{13} (2004), no. 2, 193--209.

\bibitem[Ke]{Ke} R. Kellerhals, \emph{On the volume of hyperbolic polyhedra}, Math.  Ann. \textbf{285} (1989), 541--569.

\bibitem[KM]{KM} A. Kolpakov, A. Mednykh, {\em Spherical structures on torus knots and links}, Sib. Math. J. \textbf{50} (2009), 856--866.

\bibitem[Ko]{Ko} S. Kojima, {\em Deformations of hyperbolic 3-cone-manifolds}, J. Differential Geom. \textbf{49} (1998) 469--516.
  
\bibitem[MPL]{MPL} M. Macasieb, K. Petersen and R. van Luijk, {\em On character varieties of two-bridge knot groups} Proc. Lond. Math. Soc. (3) \textbf{103} (2011), no. 3, 473--507.

\bibitem[Me]{Me}  A. Mednykh, {\em Volumes of two-bridge cone manifolds in spaces of constant curvature}, Transform. Groups \textbf{26} (2021), no.2, 601--629.

\bibitem[MR]{MR} A. Mednykh and A. Rasskazov, {\em Volumes and degeneration of cone-structures on the figure-eight knot}, Tokyo J. Math. \textbf{29} (2006) 445--464.

\bibitem[MT]{MT} T. Morifuji and A. Tran, {\em Twisted Alexander polynomials of two-bridge knots for parabolic representations}, Pacific J. Math. \textbf{269} (2014), no. 2, 433--451.

\bibitem[MV]{MV} A. Mednykh and A. Vesnin,  \emph{Hyperbolic volumes of Fibonacci manifolds}, Siberian Math. J. \textbf{36}:2 (1995), 235--245.

\bibitem[Po]{Po} J. Porti, {\em Spherical cone structures on 2-bridge knots and links}, Kobe J. Math. \textbf{21} (2004) 61--70.

\bibitem[Sn]{Sn} M. Culler, N. Dunfield, M. Goerner, and J. Weeks,
\textit{Snap{P}y, a computer program for studying the geometry and topology of $3$-manifolds}, available at \url{http://snappy.computop.org}.

\bibitem[Ri]{Ri} R. Riley, {\em Nonabelian representations of 2-bridge knot groups}, Quart. J. Math. Oxford Ser. (2) \textbf{35}
(1984), 191--208.

\bibitem[Tr]{Tr} A. Tran, {\em Volumes of hyperbolic double twist knot cone-manifolds}, J. Knot Theory Ramifications \textbf{26} (2017), no. 11, 1750068, 14 pp.

\bibitem[Vi1]{Vi1} E. Vinberg, \emph{The volume of polyhedra on a sphere and in Lobachevsky space}, Amer. Math. Soc. Transl. Ser.~2, \textbf{148} (1991), 15--27.

\bibitem[Vi2]{Vi2} E. Vinberg, \emph{Volumes of non-Euclidean polyhedra}, Russian Math.\ Surveys \textbf{48}:2 (1993), 15--45.


\end{thebibliography}
\end{document}